\documentclass[a4paper,12pt]{amsart}

\usepackage[english]{babel}
\usepackage{amsthm}
\usepackage{amssymb}
\usepackage{hyperref}

\newcommand{\Set}[1]{\left\{\, #1 \,\right\}}

\newcommand{\Span}[1]{\langle\, #1 \,\rangle}

\newcommand{\Size}[1]{\lvert #1 \rvert}

\DeclareMathOperator{\GL}{GL}
\DeclareMathOperator{\GF}{GF}
\DeclareMathOperator{\Aut}{Aut}
\DeclareMathOperator{\End}{End}

\renewcommand{\phi}[0]{\varphi}
\renewcommand{\theta}[0]{\vartheta}
\renewcommand{\epsilon}[0]{\varepsilon}

\newcommand{\F}{\text{$\mathbf{F}$}}

\theoremstyle{plain}

\newtheorem{dummy}{Dummy}
\numberwithin{dummy}{section}
\numberwithin{equation}{section}

\newtheorem{theorem}[dummy]{Theorem}
\newtheorem{lemma}[dummy]{Lemma}

\theoremstyle{definition}

\newtheorem{assumption}[dummy]{Assumption}

\theoremstyle{remark}

\newcommand{\FMEO}[1]{}

\begin{document}

\bibliographystyle{amsalpha}

\date{8 July 2014, 12:47 CEST --- Version 2.03%%
}

\title[Central Automorphisms]
{A simple construction for a class of $p$-groups\\
with all of their automorphisms central}

\author{A.~Caranti}

\address[A.~Caranti]%
 {Dipartimento di Matematica\\
  Universit\`a degli Studi di Trento\\
  via Sommarive 14\\
  I-38123 Trento\\
  Italy} 

\email{andrea.caranti@unitn.it} 

\urladdr{http://science.unitn.it/$\sim$caranti/}

\begin{abstract}
  We exhibit a  simple construction, based on elementary linear
  algebra, for a class  of examples of finite 
  $p$-groups of  nilpotence class $2$  all of whose  automorphisms are
  central.
\end{abstract}

\keywords{finite $p$-groups, automorphisms, central automorphisms,
  en\-do\-mor\-phisms} 

\thanks{The  author gratefully acknowledges the support of the
  Department of Mathematics, University of Trento. The author is a 
  member of GNSAGA---Indam.}
  
\subjclass[2010]{20D15 20D45}

\maketitle

\thispagestyle{empty}

\section{Introduction}

In  June 2014,  Marc van Leeuwen~\cite{MvL}  inquired on  Mathematics
StackExchange whether there  is a group $P$ with an  element $a \in P$
such  that  there is no  automorphism  of  $P$  taking  $a$  to  its
inverse. 

In our  answer, we noted  that an example  was provided by any  of the
many constructions in the literature~\cite{HeLie, JoKo, Ea, H79, CaLe,
  Ca83, E-group, Mo94, Mo95}  of finite $p$-groups of nilpotence class
two in which all automorphisms are central, for $p$ an odd prime. For,
if $P$ is such a group, and $a \in P \setminus Z(P)$, then an image of
$a$ under automorphisms  is of the form $a z$, with  $z \in Z(P)$.  If
$a z = a^{-1}$,  then $a^{2} \in Z(P)$, and thus $a  \in Z(P)$, as $p$
is odd.

Marc van Leuween commented that  ``indeed giving a concrete example is
not so easy''. This made us realize that examples of finite $p$-groups
in  which all  automorphisms  are central,  although not  conceptually
difficult,  usually  rely  on  a  fair  amount  of  calculations  with
generators and relations. The goal of this paper is to give a class of
examples of such groups for which  calculations  can be kept to  a
minimum, whereas  a 
central role is  played by  linear  algebra.   

The examples are based on  one of the cases of~\cite[Section 4]{Ca83}.
They  are  constructed  according  to  the  linear  algebra  techniques
employed in~\cite{HeLie, DaHe, Ca83, GPS}, as described in~\cite{TAT},
which  we  review   in  Section~\ref{sec:prelim}.   The  examples themselves  are
presented in Section~\ref{sec:main}, while in Section~\ref{sec:endo}
we mention an extension to endomorphisms.

\section{Preliminaries}\label{sec:prelim}

Let $P$ be a group. Since the centre $Z(P)$ is a characteristic
subgroup of $P$, there is a natural morphism
$\Aut(P) \to \Aut(P/Z(P))$ whose kernel $\Aut_{c}(P)$ consists of the
central automorphisms of $P$, that is, those automorphisms of $P$ that
take every $a \in P$ to an element of $a Z(P)$. 

We review the setup of~\cite{TAT}. Let $V$ be a vector space of
dimension $n+1$ over the field $\F = \GF(p)$, where $p$ is a
prime. Let $W = \Lambda^{2} V$ be the exterior square of $V$. If  $f :
V \to W$ is a linear map, we will consider the group $G$ of
the elements of $\GL(V)$ that commute with $f$, 
\begin{equation}\label{eq:G}
  G = \Set {g \in \GL(V) : (v g) f = (v f) \widehat{g}\text{, for all
      $v \in V$}},
\end{equation}
where $\widehat{g}$ is the automorphism of $W$ induced by $g$.
Note that we write maps on the right, so our vectors are row vectors.

Choose now a basis $v_{0}, v_{1}, \dots, v_{n}$ of $V$, and the
corresponding basis $v_{j} \wedge v_{k}$ of $W$, for $j < k$. Write
$f$ in coordinates, that is, 
\begin{equation*}
v_{i} f = \sum_{j < k} a_{i,j,k} \cdot v_{j} \wedge v_{k}.
\end{equation*}
If $p$ is odd, we can construct a finite $p$-group $P$ via the
following presentation 
\begin{equation}\label{eq:pres}
  \begin{aligned}
    P
    =
    \Span{x_{0}, x_{1}, \dots, x_{n}
      :\
      &\text{$[[x_{i}, x_{j}], x_{k}] = 1$ for all $i, j, k$, }
      \\&\text{$x_{i}^{p} = \prod_{j < k} [x_{j}, x_{k}]^{a_{i,j,k}}$
        for all $i$,}
      \\&\text{$[x_{i}, x_{j}]^{p} = 1$ for all $i, j$}
    }.
  \end{aligned}
\end{equation}
Note  that here the  third line  of relations  follows from  the first
two.  In fact the  first two lines of relations  say  that commutators
and $p$-th  powers of 
generators  are central, so  that we  have $1  = [x_{i}^{p},  x_{j}] =
[x_{i}, x_{j}]^{p}$, as $x_{i}$ commutes with $[x_{i}, x_{j}]$.
%% \begin{equation}\label{eq:pres}
%%   P = \Span{x_{0}, x_{1}, \dots, x_{n} : x_{i}^{p} = \prod_{j,k}
%%   [x_{j}, x_{k}]^{a_{i,j,k}}, [x_{i}, x_{j}, x_{k}] = 1}. 
%% \end{equation}

We have that $P$ is a group of nilpotence class two and order
$\Size{P} = p^{n+1 + \binom{n+1}{2}}$, with $P' = \Phi(P) = Z(P)$ of
order $p^{\binom{n+1}{2}}$, and $P/P'$ of order $p^{n+1}$. Moreover
$P^{p}$ has order $p^{\dim(V f)}$. 

If $p = 2$, we appeal to an idea of Zurek~\cite{Zurek}, and modify
\eqref{eq:pres} replacing 
%%%``$x_{i}^{p} = $'' with ``$x_{i}^{4} = $'', 
$p$-th powers of generators with $4$-th powers. 
The presentation thus becomes
\begin{equation*}\label{eq:pres2}
  \begin{aligned}
    P
    =
    \Span{x_{0}, x_{1}, \dots, x_{n}
      :\
      &\text{$[[x_{i}, x_{j}], x_{k}] = 1$ for all $i, j, k$, }
      \\&\text{$x_{i}^{4} = \prod_{j < k} [x_{j}, x_{k}]^{a_{i,j,k}}$ for all $i$,}
      \\&\text{$[x_{i}, x_{j}]^{2} = 1$ for all $i, j$}
    }.
  \end{aligned}
\end{equation*}
Here we have $\Size{P} = 2^{2(n+1) +
  \binom{n+1}{2}}$, $P'$ has order $2^{\binom{n+1}{2}}$, $P/P'$ has
order $2^{2(n+1)}$, $P^{4}$ has order $2^{\dim(V f)}$, and $P' \le
\Phi(P) = Z(P)$.  This time, the relations $[x_{i}, x_{j}]^{2} = 1$
are necessary.

Now it is shown in~\cite[Section~3]{TAT} that 
\begin{theorem}\label{theorem}
  In the notation above,
  \begin{equation*}
    \Aut(P)/\Aut_{c}(P) \cong G.
  \end{equation*}
\end{theorem}
The point of this, as explained in~\cite{TAT}, is that for an
automorphism $g$ of $P/Z(P) = P/\Phi(P)$ to be induced by an automorphism
of $P$, one needs $g$ to preserve the $p$-th (respectively, $4$-th)
power relations, that is, the linear map $f$.

\section{The examples}\label{sec:main}

We will now construct  a class of linear maps $f$, as in the previous
Section, for which the group $G$ 
of~\eqref{eq:G} is $\Set{1}$. According to Theorem~\ref{theorem}, this
will provide examples of finite $p$-groups $P$ of nilpotence class $2$
with $\Aut(P) = \Aut_{c}(P)$.

Let $V$ be a vector space of dimension $n+1 \ge 4$ over $\F =
\GF(p)$, where $p$ is a prime. Fix a basis $v_{0}, v_{1}, \dots ,
v_{n}$ of $V$, and let  
\begin{equation*}
  U = \Span{v_{1}, \dots , v_{n}}.
\end{equation*}
On the exterior square $W =
\Lambda^{2} V$, consider a basis which begins with
\begin{equation*}
  v_{0} \wedge v_{1}, v_{0} \wedge v_{2}, \dots , v_{0} \wedge v_{n},
\end{equation*}
and continues with the $v_{i} \wedge v_{j}$, for $1 \le i < j \le n$.

We now make our choice for $f$.
\begin{assumption}\label{assump}
  Consider the linear map $f : V \to W$ which, with respect to the given
  bases, has blockwise matrix
  \begin{equation}\label{eq:rels}
    \begin{bmatrix}
      b & c\\
      A & 0
    \end{bmatrix},
  \end{equation}
  where $b$ is a $1 \times n$ vector, $c$ is a $1 \times
  \dbinom{n}{2}$ vector, 
  $A$ is an $n \times n$ matrix, and $0$ is an $n \times \dbinom{n}{2}$  zero
  matrix. 
  Moreover, we take
  \begin{itemize}
  \item $b, c \ne 0$, and 
  \item $A$ to be  the companion matrix \cite[p.~197]{BAI2} of  the
    minimal polynomial 
    $m$ over $\F$
    of  a primitive  element $\alpha$  of $\GF(p^{n})$.  
  \end{itemize}
\end{assumption}

We collect a few elementary facts about the matrix $A$.
\begin{lemma}\label{lemma:A}\
  \begin{enumerate}
  \item\label{item:roots} The  roots of $m$,  i.e.~the eigenvalues of $A$,  are 
    \begin{equation*}
          \alpha,
    \alpha^{p}, \dots , \alpha^{p^{n-1}}.
    \end{equation*}
  \item\label{item:order} $A$ has multiplicative order $p^{n}-1$.
  \item\label{item:powers}  $\F[A]$ is  a field  of order  $p^{n}$, and
    $\F[A] = \Set{0} \cup \Set{A^{i} : 0 \le i < p^{n}-1}$.
  \item\label{item:simple} $\F^{n}$ is a one-dimensional
    $\F[A]$-vector space.
  \item\label{item:cent} The centralizer 
    \begin{equation*}
      C_{\End(\F^{n})}(A)
    \end{equation*}
    of $A$ in $\End(\F^{n})$ is $\F[A]$.
  \end{enumerate}
\end{lemma}

\begin{proof}
  \eqref{item:roots} follows from the fact that $\alpha$ is a root of
  $m$, and $m$ is irreducible in $\F[x]$, of degree $n$.

  \eqref{item:order} follows immediately from the previous point.

  \eqref{item:powers} follows from $\F[A] \cong \F[x]/(m)$, and
  \eqref{item:order}. 

  \eqref{item:simple} follows from the fact that $A$ is a companion matrix, and thus $\F^{n}$ is a cyclic $\F[A]$-module.
  
  \eqref{item:cent} now follows from the previous point, as the given
  centralizer is the ring of endomorphisms of the $\F[A]$-vector space
  $\F^{n}$.
  %% , since $\F^{n}$ is a simple
  %% module, an element of $C_{\End(\F^{n})}(A)$ is
  %% determined by its action on any non-zero element of $V$, so that
  %% $C_{\End(\F^{n})}(A)$ has order at most $\Size{\F^{n}} = p^{n}$, and
  %% then equals $\F[A]$, as it contains it.
\end{proof}

We now collect a few facts about $f$ and the group $G$ of~\eqref{eq:G}.
\begin{lemma}\label{lemma:omnibus}\

\begin{enumerate}
\item \label{lemma:inj} $f$ is  injective.
\item \label{Uf} 
  $
      U  f = v_{0} \wedge V = v_{0} \wedge U,
 $
and this is a subspace of $W$ of dimension $n$.
\item \label{lemma:u}  
  If $u \in U$ satisfies $u \wedge V \le V f$, then $u =
  0$.
\item \label{lemma:v0}$
    \Span{v_{0}} = \Set{x \in V : x \wedge V \le V f}.
 $
\item \label{lemma:v0inv} $\Span{v_{0}}$ is left invariant by $G$.
\item \label{lemma:U}$
  U = \Set{ x \in V : x f \in v_{0} \wedge V }.
$
\item \label{lemma:Uinv} $U$ is left invariant by $G$.
\end{enumerate}
\end{lemma}

\begin{proof}
\eqref{lemma:inj} follows  from  Assumption~\ref{assump}, since $A$ is
invertible, and  $c  \ne 0$ 
in~\eqref{eq:rels}.

The formula of~\eqref{Uf} now follows from the shape of the matrix for
$f$ in Assumption~\ref{assump}. 

To prove~\eqref{lemma:u}, let $u = c_{1} v_{1} + \dots + c_{n} v_{n}$
satisfy $u \wedge V \le V f$. We will show that $c_{1} 
  = 0$, but a similar argument yields that all $c_{i}$ have to be
  zero. Let us look at the coordinates of $u \wedge v_{2}$ and $u \wedge
  v_{3}$ with respect 
  to $v_{1} \wedge v_{2}$ and $v_{1} \wedge v_{3}$, which yield the $2
  \times 2$ matrix
  \begin{equation*}
    \begin{bmatrix}
      c_{1} & 0\\
      0 & c_{1}
    \end{bmatrix}.
  \end{equation*}
  By~\eqref{lemma:inj} and~\eqref{Uf}, the dimension of $V f / U f = V
  f/(v_{0} \wedge V)$ is
  $1$. Thus this matrix must have rank at most $1$, so that
  $c_{1}^{2} = 0$. 

To prove~\eqref{lemma:v0}, let $0 \ne x \in V$ be such that $x \wedge
V \le V f$. By~\eqref{lemma:u}, $x = c v_{0} + u$, for some $c \ne
0$, and $u \in U$. But then by~\eqref{Uf} $u \wedge V \le V f$, so
that $u = 0$ again by~\eqref{lemma:u}.

\eqref{lemma:v0inv} follows from the previous point.

\eqref{lemma:U} follows from~\eqref{lemma:inj} and~\eqref{Uf}, and
implies~\eqref{lemma:Uinv}, because of~\eqref{lemma:v0inv}. 
\end{proof}

(Note that the argument in  the proof of~\eqref{lemma:u} fails when $n
= 2$, see~\cite{DaHe}, and this is the reason we have taken $n + 1 \ge
4$.)

We can now state our main result.
\begin{theorem}
With $f$ as in  Assumption~\ref{assump}, the group $G$ of~\eqref{eq:G}
is $\Set{1}$.
\end {theorem}

\begin{proof}
\eqref{lemma:v0inv} and~\eqref{lemma:Uinv} of
Lemma~\ref{lemma:omnibus} allow us to write an element 
$g \in G$ in matrix form, with 
respect to the given basis of $V$, as
\begin{equation*}
  g = 
  \begin{bmatrix}
    \gamma & 0\\
    0      & \Delta\\
  \end{bmatrix}
\end{equation*}
where $\gamma \in \F^{\star}$, and $\Delta \in \GL(n,\F)$. By the
definition~\eqref{eq:G} 
of $G$, and Assumption~\ref{assump}, we have
\begin{equation}\label{eq:theequality}
  \begin{bmatrix}
    \gamma & 0\\
    0      & \Delta\\
  \end{bmatrix}
  \begin{bmatrix}
    b & c\\
    A & 0
  \end{bmatrix}
  =
  \begin{bmatrix}
    b & c\\
    A & 0
  \end{bmatrix}
  \begin{bmatrix}
    \gamma \Delta & 0\\
    0 & \widehat{\Delta}
  \end{bmatrix},
\end{equation}
where $\widehat{\Delta}$ is the matrix induced by $\Delta$ on $U
\wedge U$. We will only need consider the following two consequences of~\ref{eq:theequality}:
\begin{equation}\label{eq:a}
  \Delta A \Delta^{-1} = \gamma A,
\end{equation}
and
\begin{equation}\label{eq:b}
  b \Delta = b.
\end{equation}

To deal with~\eqref{eq:a}, we could appeal to~\cite{Wall}, but prefer
to give a simple direct argument.
As noted in Lemma~\ref{lemma:A}.\eqref{item:roots}, the eigenvalues of $A$ are 
\begin{equation}\label{eq:eigenA}
  \alpha, \alpha^{p}, \dots, \alpha^{p^{n-1}},
\end{equation}
with $\alpha$ a primitive element, so that those of
$\gamma A$ are  
\begin{equation*}
 \gamma \alpha, \gamma \alpha^{p}, \dots,  \gamma \alpha^{p^{n-1}}.
\end{equation*} 
By~\eqref{eq:a} we have $\gamma \alpha = \alpha^{p^{t}}$
for some $t$. If $t > 0$, then 
\begin{equation*}
  \alpha^{p^{t} - 1} = \gamma \in \GF(p)^{\star},
\end{equation*}
with $p^{t} - 1 > 0$,
so that the order
\begin{equation*}
\frac{p^{n}-1}{p-1} = 1 + p + \dots + p^{n-1}
\end{equation*}
of $\alpha$ in
$\GF(p^{n})^{\star}/\GF(p)^{\star}$ divides $p^{t} - 1 < p^{n-1}$, a
contradiction. Thus $t = 0$ and $\gamma = 1$.

It follows that $\Delta$ is in the centralizer of $A$ in $\GL(n, \F)$,
and thus, according to Lemma~\ref{lemma:A},  $\Delta$ is
a power of $A$. 

But once more since the eigenvalues of $A$ are as in~\eqref{eq:eigenA}, with
$\alpha$ a primitive element, the only
power of $A$ to have an 
eigenvalue $1$ is $1$. Since we have~\eqref{eq:b}, with $b \ne 0$ by
Assumption~\ref{assump}, we obtain that $\Delta = 1$ and thus $G = \{ 1 
\}$ as claimed.
\end{proof}

%% Note that the arguments of this Section apply even if we take $c =
%% 0$ in Assumption~\ref{assump}. In this case $f$ is not injective, and
%% $\dim(V f) = n$.  

%%%\input{Endomorphisms}
\section{Endomorphisms}\label{sec:endo}

The arguments of the previous Section can be slightly extended to show
that the set
\begin{equation*}
  G = \Set {g \in \End(V) : (v g) f = (v f) \hat{g}\text{, for all $v
      \in V$}} 
\end{equation*}
of the endomorphisms of $V$ that commute with $f$
consists of $0$ and $1$. Now in our examples $P$ the centre 
$Z(P)$ is fully invariant, as it equals $\Phi(P)$, 
and thus an immediate extension of
Theorem~\ref{theorem} yields that an endomorphism of
$P$ either  maps $P$
into $Z(P)$, or is a central automorphism,
so that $P$ is an E-group~\cite{Fau, Mal77, Mal80,
  E-group, CFdG}, that is, a 
group in which each element commutes with all of its images under
endomorphisms.   

To prove this, we proceed as in the previous Section, except that we make no
assumptions on $\gamma$ and $\Delta$. We have
from~\eqref{eq:theequality} 
\begin{equation}\label{eq:aa}
  \Delta A  = \gamma A \Delta.
\end{equation}
%% and
%% \begin{equation}\label{eq:bb}
%%   b \Delta = b.
%% \end{equation}
If $\gamma  = 0$,  then $\Delta =  0$. If  $\gamma \ne 0$,  it follows
from~\ref{eq:aa}  that $\ker(\Delta)$  is $A$-invariant, that is, a
$\F[A]$-vector subspace of the one-dimensional $\F[A]$-vector space
$\F^{n}$. Therefore we have either $\ker(\Delta) =
\Set{0}$, that is, $\Delta$ is invertible, and we proceed as above, or
$\ker(\Delta) =  \F^{n}$, that  is, $\Delta =  0$. Now~\eqref{eq:theequality}
yields $\gamma  b = 0$, a contradiction  to $\gamma \ne 0$  and $b \ne
0$.

%%%\bibliography{Refs}

\begin{thebibliography}{CFdG87}

\bibitem[Car83]{Ca83}
A.~Caranti, \emph{Automorphism groups of {$p$}-groups of class {$2$} and
  exponent {$p^{2}$}: a classification on {$4$} generators}, Ann. Mat. Pura
  Appl. (4) \textbf{134} (1983), 93--146. \MR{736737 (85j:20036)}

\bibitem[Car85]{E-group}
\bysame, \emph{Finite {$p$}-groups of exponent {$p^2$} in which each element
  commutes with its endomorphic images}, J. Algebra \textbf{97} (1985), no.~1,
  1--13. \MR{812164 (87b:20029)}

\bibitem[Car13]{TAT}
\bysame, \emph{A module-theoretic approach to abelian automorphism groups},
  Israel J. Math. (2013), accepted for publication.

\bibitem[CFdG87]{CFdG}
A.~Caranti, S.~Franciosi, and F.~de~Giovanni, \emph{Some examples of infinite
  groups in which each element commutes with its endomorphic images}, Group
  theory ({B}ressanone, 1986), Lecture Notes in Math., vol. 1281, Springer,
  Berlin, 1987, pp.~9--17. \MR{921685 (88m:20067)}

\bibitem[CL82]{CaLe}
A.~Caranti and P.~Legovini, \emph{On finite groups whose endomorphic images are
  characteristic subgroups}, Arch. Math. (Basel) \textbf{38} (1982), no.~5,
  388--390. \MR{666909 (84b:20022)}

\bibitem[DH75]{DaHe}
G.~Daues and H.~Heineken, \emph{Dualit\"aten und {G}ruppen der {O}rdnung
  {$p^{6}$}}, Geometriae Dedicata \textbf{4} (1975), no.~2/3/4, 215--220.
  \MR{0401907 (53 \#5733)}

\bibitem[Ear75]{Ea}
B.~E. Earnley, \emph{On finite groups whose group of automorphisms is abelian},
  Ph.D. thesis, Wayne State University, 1975, Dissertation Abstracts, V. 36, p.
  2269 B.

\bibitem[Fau71]{Fau}
R.~Faudree, \emph{Groups in which each element commutes with its endomorphic
  images}, Proc. Amer. Math. Soc. \textbf{27} (1971), 236--240. \MR{0269737 (42
  \#4632)}

\bibitem[GPS11]{GPS}
S.~P. Glasby, P.~P. P{\'a}lfy, and Csaba Schneider, \emph{{$p$}-groups with a
  unique proper non-trivial characteristic subgroup}, J. Algebra \textbf{348}
  (2011), 85--109. \MR{2852233 (2012j:20059)}

\bibitem[Hei80]{H79}
Hermann Heineken, \emph{Nilpotente {G}ruppen, deren s\"amtliche {N}ormalteiler
  charakteristisch sind}, Arch. Math. (Basel) \textbf{33} (1979/80), no.~6,
  497--503. \MR{570484 (81h:20023)}

\bibitem[HL74]{HeLie}
Hermann Heineken and Hans Liebeck, \emph{The occurrence of finite groups in the
  automorphism group of nilpotent groups of class {$2$}}, Arch. Math. (Basel)
  \textbf{25} (1974), 8--16. \MR{0349844 (50 \#2337)}

\bibitem[Jac85]{BAI2}
Nathan Jacobson, \emph{Basic algebra. {I}}, second ed., W. H. Freeman and
  Company, New York, 1985. \MR{780184 (86d:00001)}

\bibitem[JK75]{JoKo}
D.~Jonah and M.~Konvisser, \emph{Some non-abelian {$p$}-groups with abelian
  automorphism groups}, Arch. Math. (Basel) \textbf{26} (1975), 131--133.
  \MR{0367059 (51 \#3301)}

\bibitem[Mal77]{Mal77}
J.~J. Malone, \emph{More on groups in which each element commutes with its
  endomorphic images}, Proc. Amer. Math. Soc. \textbf{65} (1977), no.~2,
  209--214. \MR{0447351 (56 \#5664)}

\bibitem[Mal80]{Mal80}
\bysame, \emph{A nonabelian {$2$}-group whose endomorphisms generate a ring,
  and other examples of {$E$}-groups}, Proc. Edinburgh Math. Soc. (2)
  \textbf{23} (1980), no.~1, 57--59. \MR{582023 (81m:20057)}

\bibitem[Mor94]{Mo94}
Marta Morigi, \emph{On {$p$}-groups with abelian automorphism group}, Rend.
  Sem. Mat. Univ. Padova \textbf{92} (1994), 47--58. \MR{1320477 (96k:20037)}

\bibitem[Mor95]{Mo95}
\bysame, \emph{On the minimal number of generators of finite non-abelian
  {$p$}-groups having an abelian automorphism group}, Comm. Algebra \textbf{23}
  (1995), no.~6, 2045--2065. \MR{1327121 (96a:20028)}

\bibitem[vL14]{MvL}
Marc van Leeuwen, \emph{Can a group have a subset that is stable under all
  automorphisms, but not under inverse?}, Mathematics StackExchange, June 2014,
  http://goo.gl/l87A5t.

\bibitem[Wal80]{Wall}
G.~E. Wall, \emph{Conjugacy classes in projective and special linear groups},
  Bull. Austral. Math. Soc. \textbf{22} (1980), no.~3, 339--364. \MR{601642
  (82d:20041)}

\bibitem[Zur82]{Zurek}
Gerhard Zurek, \emph{A comment on a work by {H}. {H}eineken and {H}. {L}iebeck:
  ``{T}he occurrence of finite groups in the automorphism group of nilpotent
  groups of class {$2$}''\ [{A}rch. {M}ath. ({B}asel) {\bf 25} (1974), 8--16;\
  {MR} {\bf 50} \#2337]}, Arch. Math. (Basel) \textbf{38} (1982), no.~3,
  206--207. \MR{656185 (83k:20034)}

\end{thebibliography}
\providecommand{\bysame}{\leavevmode\hbox to3em{\hrulefill}\thinspace}
\providecommand{\MR}{\relax\ifhmode\unskip\space\fi MR }
% \MRhref is called by the amsart/book/proc definition of \MR.
\providecommand{\MRhref}[2]{%
  \href{http://www.ams.org/mathscinet-getitem?mr=#1}{#2}
}
\providecommand{\href}[2]{#2}

\end{document}